\documentclass[11pt]{article}
\usepackage{amsfonts, amsmath, amsthm, amssymb, mathtools, enumitem, tikz, hyperref, caption}
\usepackage{eucal}
\usepackage[noabbrev,capitalize]{cleveref}

\definecolor{darkgray}{RGB}{64,64,64}
\definecolor{litegray}{RGB}{192,192,192}
\definecolor{green}{HTML}{0F9D58}
\definecolor{red}{HTML}{DB4437}
\tikzstyle{vertex}=[circle, draw, fill=red, inner sep=0pt, minimum width=5pt]
\tikzstyle{vtx}=[circle, draw, fill=litegray, inner sep=0pt, minimum width=5pt]

\usepackage{microtype, fullpage}
\newtheorem{theorem}{Theorem}[section]
\newtheorem{lemma}[theorem]{Lemma}
\newtheorem{proposition}[theorem]{Proposition}
\newtheorem{corollary}[theorem]{Corollary}

\theoremstyle{definition}

\theoremstyle{remark}

\newtheorem*{claim*}{Claim}

\newcommand{\PP}{\mathcal{P}}
\newcommand{\SSS}{\mathcal{S}}

\title{A bound for $1$-cross intersecting set pair systems}

\author{
  Ron Holzman\footnote{Department of Mathematics, Technion -- Israel Institute of Technology, Technion City, Haifa 3200003, Israel. Email: {\tt holzman@technion.ac.il}. Research done during a visit at the Department of Mathematics, Princeton University, partially supported by the H2020-MSCA-RISE project CoSP--GA No. 823748.}
}

\date{}

\begin{document}

\maketitle

\begin{abstract}
A well-known result of Bollob\'as says that if $\{(A_i, B_i)\}_{i=1}^m$ is a set pair system such that $|A_i| \le a$ and $|B_i| \le b$ for $1 \le i \le m$, and $A_i \cap B_j \ne \emptyset$ if and only if $i \ne j$, then $m \le {a+b \choose a}$. F\"uredi, Gy\'arf\'as and Kir\'aly recently initiated the study of such systems with the additional property that $|A_i \cap B_j| = 1$ for all $i \ne j$. Confirming a conjecture of theirs, we show that this extra condition allows an improvement of the upper bound (at least) by a constant factor.
\end{abstract}

\section{Introduction} \label{sec:introduction}

A system of $m$ pairs of finite sets $\{(A_i, B_i)\}_{i=1}^m$ is \emph{cross intersecting} if the following two conditions hold:
\begin{align}
A_i \cap B_i &= \emptyset \textrm{  for every  } 1 \le i \le m, \label{eq:disjoint}\\
A_i \cap B_j &\ne \emptyset \textrm{  for every  } 1 \le i \ne j \le m. \label{eq:intersecting}
\end{align}
The question is how large can the size $m$ of such a system be, if $|A_i| \le a$ and $|B_i| \le b$ for every $1 \le i \le m$. The \emph{standard example} of such a system is constructed by taking a ground set of $a+b$ elements, and forming all pairs of complementary sets $(A_i, B_i)$ with $|A_i| = a, |B_i| = b$; its size is $m = {a+b \choose a}$. The classical result of Bollob\'as~\cite{B} is that this is the largest size possible.

Several different proofs have been given for this result. Quite a few extensions and variants of it have been studied, and there are many applications in extremal set theory and beyond. For more information, we refer to the surveys~\cite{F, T94, T96}.

Actually, the result of Bollob\'as is more refined, taking into account sets $A_i$ and $B_i$ of different sizes. Here is the statement.

\begin{theorem}[\cite{B}] \label{thm:bollobas}
Let $\{(A_i, B_i)\}_{i=1}^m$ be a cross intersecting set pair system with $|A_i| \le a_i$ and $|B_i| \le b_i$ for every $1 \le i \le m$. Then
\[ \sum_{i=1}^m \frac{1}{{a_i + b_i \choose a_i}} \le 1, \]
and equality holds only if there exist $a$ and $b$ such that $|A_i| = a$ and $|B_i| = b$ for every $1 \le i \le m$, and the system is the standard example.
\end{theorem}

A set pair system $\{(A_i, B_i)\}_{i=1}^m$ is $1$-\emph{cross intersecting} if condition~(\ref{eq:disjoint}) holds, and condition~(\ref{eq:intersecting}) holds in the stronger form
\begin{equation} \label{eq:one}
|A_i \cap B_j| = 1 \textrm{  for every  } 1 \le i \ne j \le m.
\end{equation}
F\"uredi, Gy\'arf\'as and Kir\'aly~\cite{FGK} recently introduced this concept. They sought an improvement, under this strengthening, of the upper bound $m \le {a+b \choose a}$, where $|A_i| \le a$ and $|B_i| \le b$ for every $1 \le i \le m$. Clearly, when $a$ or $b$ is $1$ the strengthening has no effect. They obtained a sharp upper bound whenever $a$ or $b$ is $2$; in this case, when the other parameter is large, the strengthening improves the bound by a factor of roughly $\frac{1}{2}$. For illustration and future reference, we state the special case $a=b=2$ of their result.

\begin{proposition}[\cite{FGK}] \label{prop:fgk}
Let $\{(A_i, B_i)\}_{i=1}^m$ be a $1$-cross intersecting set pair system with $|A_i| \le 2$ and $|B_i| \le 2$ for every $1 \le i \le m$. Then $m \le 5$, and equality holds only if $\{A_i\}_{i=1}^5$ and $\{B_i\}_{i=1}^5$ form two complementary $5$-cycles (that is, the vertices may be written as $0,1,2,3,4 \mod 5$, so that $A_i = \{i, i+1\}$ and $B_i = \{i-1, i+2\}$ for $1 \le i \le 5$).
\end{proposition}

Beyond the cases mentioned above, when $a$ and $b$ are both greater than $2$, the maximum size of a $1$-cross intersecting system with $|A_i| \le a$ and $|B_i| \le b$, is not known. Here we prove the following upper bound which, like Theorem~\ref{thm:bollobas}, accounts for different set sizes.

\begin{theorem} \label{thm:mine}
Let $a_i, b_i \ge 2$ for $1 \le i \le m$, and let $\{(A_i, B_i)\}_{i=1}^m$ be a $1$-cross intersecting set pair system with $|A_i| \le a_i$ and $|B_i| \le b_i$ for every $1 \le i \le m$. Then
\[ \sum_{i=1}^m \frac{1}{{a_i + b_i \choose a_i}} \le \frac{29}{30}. \]
\end{theorem}

\begin{corollary} \label{cor:size}
Let $a,b \ge 2$ and let $\{(A_i, B_i)\}_{i=1}^m$ be a $1$-cross intersecting set pair system with $|A_i| \le a$ and $|B_i| \le b$ for every $1 \le i \le m$. Then $m \le \frac{29}{30} {a+b \choose a}$.
\end{corollary}

The case $a=b$ of the corollary confirms a conjecture of F\"uredi, Gy\'arf\'as and Kir\'aly~\cite{FGK}: they postulated the existence of a positive $\varepsilon$ such that $m \le (1 - \varepsilon) {2n \choose n}$ for every $1$-cross intersecting system $\{(A_i, B_i)\}_{i=1}^m$ with $|A_i| \le n$ and $|B_i| \le n$ for every $1 \le i \le m$, and every $n \ge 2$. It seems likely that our constant $\frac{29}{30}$ can be improved to $\frac{5}{6}$, which would be best possible in view of Proposition~\ref{prop:fgk}. Another plausible conjecture in~\cite{FGK} says that the constant can be made arbitrarily small if $n$ is large enough. One could even conjecture an upper bound of the form $C^n$ on the size $m$ of a $1$-cross intersecting system with sets of size $n$, where $C$ is a constant less than $4$. The best construction known (see~\cite{FGK}) shows that $C$ must be at least $\sqrt{5}$.

As pointed out in~\cite{FGK}, the notion of $1$-cross intersecting set pair systems is closely related to the much studied topic of clique and biclique partitions of graphs. We end the introduction with a reformulation of our result in that terminology. A \emph{biclique partition} of a graph $G = (V, E)$ is a partition $\PP$ of its edge set $E$ into edge sets of complete bipartite graphs (bicliques). For such a partition $\PP$ and a vertex $v \in V$, we denote by $\PP[v]$ the family of bicliques in $\PP$ containing the vertex $v$. The \emph{thickness} of $\PP$ at $v$, denoted by $t_{\PP}(v)$, is $|\PP[v]|$. The graph $B_{2m}$ is obtained from the complete bipartite graph $K_{m,m}$ by removing a perfect matching. That is, $B_{2m} = (V, E)$ where
\begin{align*}
V & = \{x_1, x_2, \ldots , x_m\} \cup \{y_1, y_2, \ldots , y_m\},\\
E & = \{\{x_i, y_j\}:\, 1 \le i \ne j \le m\}.
\end{align*}

Given a biclique partition $\PP$ of the graph $B_{2m}$, the set pair system $\{(\PP[x_i], \PP[y_i])\}_{i=1}^m$ is $1$-cross intersecting. This observation leads to the following reformulation of our main result.

\begin{corollary} \label{cor:graph}
Let $a_i, b_i \ge 2$ for $1 \le i \le m$, and let $\PP$ be a biclique partition of the graph $B_{2m}$ with thickness $t_{\PP}(x_i) \le a_i$ and $t_{\PP}(y_i) \le b_i$ for every $1 \le i \le m$. Then
\[ \sum_{i=1}^m \frac{1}{{a_i + b_i \choose a_i}} \le \frac{29}{30}. \]
In particular, if $a, b \ge 2$ and $B_{2m}$ admits a biclique partition with thickness at most $a$ at every $x_i$ and at most $b$ at every $y_i$, then $m \le \frac{29}{30} {a+b \choose a}$.
\end{corollary}

\section{Proof} \label{sec:proof}
In preparation for the proof of Theorem~\ref{thm:mine}, we start with some notations and lemmas.

Let $\SSS = \{(A_i, B_i)\}_{i \in I}$ be a set pair system, where $I$ is a finite index set. We write $V(\SSS) = \bigcup_{i \in I} (A_i \cup B_i)$ for the ground set of the system $\SSS$. Given a subset $R \subseteq V(\SSS)$, the \emph{reduction} of $\SSS$ by $R$, denoted by $\SSS - R$, is the set pair system
\[ \SSS - R = \{(A_i \setminus R, B_i \setminus R)\}_{i \in I}. \]
We may use this notation also when the set $R$ is not contained in $V(\SSS)$, with the understanding that $\SSS - R = \SSS - (R \cap V(\SSS))$. The following lemma, stated here for future reference, is an immediate consequence of the definitions.

\begin{lemma} \label{lem:reduction}
Let $\SSS = \{(A_i, B_i)\}_{i \in I}$ be a $1$-cross intersecting set pair system with ground set $V(\SSS)$. Let $R \subseteq V(\SSS)$ be such that there are no $v \in R$ and $i \ne j \in I$ with $v \in A_i \cap B_j$. Then the reduced system $\SSS - R$ is a $1$-cross intersecting set pair system with ground set $V(\SSS) \setminus R$.
\end{lemma}

Let $\SSS = \{(A_i, B_i)\}_{i \in I}$ be a set pair system. We use the short-hand notation $\Sigma(\SSS)$ for the sum
\[ \Sigma(\SSS) = \sum_{i \in I} \frac{1}{{|A_i| + |B_i| \choose |A_i|}}. \]
If $J \subseteq I$ is a subset of the index set, we write $\SSS[J]$ for the corresponding subsystem of $\SSS$, i.e.,
\[ \SSS[J] = \{(A_i, B_i)\}_{i \in J}. \]

The next lemma is the essence of the original proof of Theorem~\ref{thm:bollobas}, by induction on the size of the ground set. We state it here in the form that we will use, and provide the proof for completeness. For a system $\SSS = \{(A_i, B_i)\}_{i \in I}$ and an element $v \in V(\SSS)$ we consider two subsystems:
\[ \SSS[I^A_{\bar{v}}] \textrm{  where  } I^A_{\bar{v}} = \{i \in I:\, v \notin A_i\}, \]
and similarly
\[ \SSS[I^B_{\bar{v}}] \textrm{  where  } I^B_{\bar{v}} = \{i \in I:\, v \notin B_i\}. \]
The reductions of these subsystems by $\{v\}$ appear in the lemma.

\begin{lemma} \label{lem:induction}
Let $\SSS = \{(A_i, B_i)\}_{i \in I}$ be a set pair system such that $A_i \ne \emptyset, B_i \ne \emptyset, A_i \cap B_i = \emptyset$ for every $i \in I$. Then
\[ \Sigma(\SSS) = \frac{1}{|V(\SSS)|} \sum_{v \in V(\SSS)} \Sigma(\SSS[I^A_{\bar{v}}] - \{v\}) \le \max_{v \in V(\SSS)} \Sigma(\SSS[I^A_{\bar{v}}] - \{v\}), \]
and similarly
\[ \Sigma(\SSS) = \frac{1}{|V(\SSS)|} \sum_{v \in V(\SSS)} \Sigma(\SSS[I^B_{\bar{v}}] - \{v\}) \le \max_{v \in V(\SSS)} \Sigma(\SSS[I^B_{\bar{v}}] - \{v\}). \]
\end{lemma}
\begin{proof} By symmetry, it suffices to prove the first statement. The inequality is obvious: the average cannot exceed the maximum. We only need to prove the identity
\[ \Sigma(\SSS) = \frac{1}{|V(\SSS)|} \sum_{v \in V(\SSS)} \Sigma(\SSS[I^A_{\bar{v}}] - \{v\}). \]
Fix an $i \in I$ and consider its contribution to the sum on the right-hand side. For every $v \in B_i$, it contributes $\frac{1}{{|A_i| + |B_i| - 1 \choose |A_i|}}$, and for every $v \in V(\SSS) \setminus (A_i \cup B_i)$ it contributes $\frac{1}{{|A_i| + |B_i| \choose |A_i|}}$. Thus the total contribution is
\[ \frac{|B_i|}{{|A_i| + |B_i| - 1 \choose |A_i|}} + \frac{|V(\SSS)| - |A_i| - |B_i|}{{|A_i| + |B_i| \choose |A_i|}} = \frac{|V(\SSS)|}{{|A_i| + |B_i| \choose |A_i|}}. \]
Upon dividing by $|V(\SSS)|$, this equals the contribution of $i$ to $\Sigma(\SSS)$, which proves the identity.
\end{proof}

We will also need the following simple bound on the ratio between certain binomial coefficients.

\begin{lemma} \label{lem:ratio}
For $a, b \ge 2$ we have
\[ \frac{{a+b-2 \choose a-1}}{{a+b \choose a}} \le \frac{1}{3}. \]
Moreover, the upper bound may be improved to $\frac{3}{10}$ unless $a=b=2$.
\end{lemma}
\begin{proof} The ratio in the lemma is equal to $\frac{ab}{(a+b)(a+b-1)}$. The upper bound of $\frac{1}{3}$ follows from
\[ (a+b)^2 \ge 4ab \ge 3ab + a + b, \]
where the second inequality uses $a, b \ge 2$. For the upper bound of $\frac{3}{10}$, the same argument works if we can replace $3ab$ by $\frac{10}{3} ab$ in the second inequality. This requires $\frac{2}{3} ab \ge a+b$ which indeed holds if $a, b \ge 3$. For the remaining case where, say, $a=2$ and $b \ge 3$, we observe that
\[ \frac{2b}{(b+2)(b+1)} \le \frac{3}{10} \,\,\, \Leftrightarrow \,\,\, 3(b+2)(b+1) \ge 20b \,\,\, \Leftrightarrow \,\,\, (3b-2)(b-3) \ge 0, \]
and the last inequality indeed holds for $b \ge 3$.
\end{proof}

We are now ready to prove Theorem~\ref{thm:mine}. The idea is to use induction on the size of the ground set, as Bollob\'as did in his proof of Theorem~\ref{thm:bollobas}. The induction step works the same way, and the gain should come from the induction base. There are, however, set pair systems in which some sets have size $2$ and others are larger. Such systems cannot be handled directly by reducing to smaller systems satisfying the assumptions of the theorem, and this requires some careful case analysis.

\begin{proof}[Proof of Theorem~\ref{thm:mine}] We may assume that the inequalities $|A_i| \le a_i$ and $|B_i| \le b_i$ are equalities. Indeed, any set for which the inequality is strict may be augmented using new elements that belong only to that set, without affecting the assumptions or the conclusion of the theorem.

Thus, if the theorem does not hold, then there exists a $1$-cross intersecting set pair system $\SSS = \{(A_i, B_i)\}_{i \in I}$ with $|A_i| \ge 2$ and $|B_i| \ge 2$ for every $i \in I$, such that (using our short-hand notation) $\Sigma(\SSS) > \frac{29}{30}$. Among all such counterexamples, we consider a system $\SSS$ with the smallest ground set $V(\SSS)$, and derive a contradiction.

Assume first that $|A_i| \ge 3$ for every $i \in I$. By Lemma~\ref{lem:reduction}, each reduced system of the form $\SSS[I^B_{\bar{v}}] - \{v\}$ for $v \in V(\SSS)$ is $1$-cross intersecting. By our assumption, all sets in such a reduced system are of size at least $2$, and the ground set is smaller than $V(\SSS)$. By the minimality of the system $\SSS$, we conclude that $\Sigma(\SSS[I^B_{\bar{v}}] - \{v\}) \le \frac{29}{30}$ for every $v \in V(\SSS)$. Now the second statement in Lemma~\ref{lem:induction} implies that $\Sigma(\SSS) \le \frac{29}{30}$, contradicting our assumption. Similarly, if $|B_i| \ge 3$ for every $i \in I$, then we can use the reduced systems $\SSS[I^A_{\bar{v}}] - \{v\}$ and get a contradiction from the first statement in Lemma~\ref{lem:induction}.

Henceforth we assume that there exist $k, \ell \in I$ with $|A_k| = |B_{\ell}| = 2$. We distinguish cases depending on whether $k$ and $\ell$ are distinct or not.

\paragraph{Case 1.}
There exist distinct $k, \ell \in I$ with $|A_k| = |B_{\ell}| = 2$.

Fix such $k$ and $\ell$. Since $|A_k \cap B_{\ell}| = 1$, we can write
\[ A_k = \{x,y\},\,\,\,B_{\ell} = \{x,z\}. \]
For any other $i \in I$, the set $A_i$ must contain exactly one of the elements $x$ and $z$, and the set $B_i$ must contain exactly one of $x$ and $y$. There are three distinct ways in which this can happen, so we get a partition of $I \setminus \{k,\ell\}$ into three (possibly empty) subsets $I_1, I_2, I_3$ as follows:
\begin{align*}
I_1 & = \{i \in I:\, x \in A_i,\, y \in B_i,\, z \notin A_i\}, \\
I_2 & = \{i \in I:\, x \in B_i,\, y \notin B_i,\, z \in A_i\}, \\
I_3 & = \{i \in I:\, x \notin A_i \cup B_i,\, y \in B_i,\, z \in A_i\}.
\end{align*}
There cannot exist both an $i \in I_2$ with $y \in A_i$ and a $j \in I_1$ with $z \in B_j$, because that would imply that $\{y,z\} \subseteq A_i \cap B_j$. Without loss of generality, we assume that $y \notin A_i$ for all $i \in I_2$, and rewrite
\[ I_2 = \{i \in I:\, x \in B_i,\, y \notin A_i \cup B_i,\, z \in A_i\}. \]
In the subsystem $\SSS[I_1]$ neither $x$ nor $y$ appears in any cross intersection. Hence, by Lemma~\ref{lem:reduction}, the reduced system $\SSS[I_1] - \{x,y\}$ is $1$-cross intersecting (actually, we only use that it is cross intersecting). By Theorem~\ref{thm:bollobas} we have
\[ \Sigma(\SSS[I_1] - \{x,y\}) = \sum_{i \in I_1} \frac{1}{{|A_i \setminus \{x\}| + |B_i \setminus \{y\}| \choose |A_i \setminus \{x\}|}} \le 1. \]
Using this and Lemma~\ref{lem:ratio}, we find that
\begin{equation} \label{eq:I1}
\Sigma(\SSS[I_1]) = \sum_{i \in I_1} \frac{1}{{|A_i| + |B_i| \choose |A_i|}} \le \frac{1}{3} \sum_{i \in I_1} \frac{1}{{|A_i| + |B_i| - 2 \choose |A_i| - 1}} \le \frac{1}{3}.
\end{equation}

Similarly, by our assumption on $I_2$, in the subsystem $\SSS[I_2 \cup I_3]$ none of $x$, $y$ and $z$ appears in any cross intersection. Hence, by Lemma~\ref{lem:reduction}, the reduced system $\SSS[I_2 \cup I_3] - \{x,y,z\}$ is $1$-cross intersecting. By Theorem~\ref{thm:bollobas} we have
\[ \Sigma(\SSS[I_2 \cup I_3] - \{x,y,z\}) = \sum_{i \in I_2 \cup I_3} \frac{1}{{|A_i| -1 + |B_i| - 1 \choose |A_i| - 1}} \le 1. \]
Again by Lemma~\ref{lem:ratio}, this implies
\begin{equation} \label{eq:I23}
\Sigma(\SSS[I_2 \cup I_3]) = \sum_{i \in I_2 \cup I_3} \frac{1}{{|A_i| + |B_i| \choose |A_i|}} \le \frac{1}{3} \sum_{i \in I_2 \cup I_3} \frac{1}{{|A_i| + |B_i| - 2 \choose |A_i| - 1}} \le \frac{1}{3}.
\end{equation}

Using~(\ref{eq:I1}), (\ref{eq:I23}) and the obvious bounds on the terms associated with $k$ and $\ell$ we can write
\begin{equation} \label{eq:total}
\Sigma(\SSS) = \frac{1}{{|A_k| + |B_k| \choose |A_k|}} + \frac{1}{{|A_{\ell}| + |B_{\ell}| \choose |A_{\ell}|}} + \Sigma(\SSS[I_1]) + \Sigma(\SSS[I_2 \cup I_3]) \le \frac{1}{6} + \frac{1}{6} + \frac{1}{3} + \frac{1}{3} = 1.
\end{equation}
If either $|B_k|$ or $|A_{\ell}|$ is at least $3$, then the bound on the corresponding term in~(\ref{eq:total}) may be decreased from $\frac{1}{6}$ to $\frac{1}{10}$, resulting in $\Sigma(\SSS) \le \frac{14}{15} < \frac{29}{30}$, a contradiction. If either $\SSS[I_1]$ or $\SSS[I_2 \cup I_3]$ contains no pair $(A_i, B_i)$ with $|A_i| = |B_i| = 2$, then Lemma~\ref{lem:ratio} allows us to replace the $\frac{1}{3}$ in either~(\ref{eq:I1}) or (\ref{eq:I23}) by $\frac{3}{10}$, which improves the upper bound in~(\ref{eq:total}) to $\Sigma(\SSS) \le \frac{29}{30}$, again a contradiction.

Thus, we may assume that $|A_i| = |B_i| = 2$ holds for at least $4$ values of $i$, namely $k$, $\ell$, a member of $I_1$, and a member of $I_2 \cup I_3$. If there is yet another such value of $i$, then our system $\SSS$ contains $5$ pairs of sets of size $2$, which by Proposition~\ref{prop:fgk} must form two complementary $5$-cycles. But then these $5$ pairs are the entire system $\SSS$, because no set (of any size) can contain exactly one vertex of every edge of a $5$-cycle. This gives $\Sigma(\SSS) = \frac{5}{6} < \frac{29}{30}$, a contradiction.

It follows that there is a unique $p \in I_1$ such that $|A_p| = |B_p| = 2$ and a unique $q \in I_2 \cup I_3$ such that $|A_q| = |B_q| = 2$. This allows us to sharpen the bounds in~(\ref{eq:I1}) and (\ref{eq:I23}) as follows. We have
\[ \Sigma(\SSS[I_1] - \{x,y\}) = \frac{1}{2} + \sum_{i \in I_1 \setminus \{p\}} \frac{1}{{|A_i| + |B_i| - 2 \choose |A_i| - 1}} \le 1, \]
and therefore by Lemma~\ref{lem:ratio}
\[ \Sigma(\SSS[I_1]) = \frac{1}{6} + \sum_{i \in I_1 \setminus \{p\}} \frac{1}{{|A_i| + |B_i| \choose |A_i|}} \le \frac{1}{6} + \frac{3}{10} \sum_{i \in I_1 \setminus \{p\}} \frac{1}{{|A_i| + |B_i| - 2 \choose |A_i| - 1}} \le \frac{1}{6} + \frac{3}{10} \cdot \frac{1}{2} = \frac{19}{60}. \]
The same argument shows that $\Sigma(\SSS[I_2 \cup I_3]) \le \frac{19}{60}$, and plugging these two improved bounds in~(\ref{eq:total}) gives $\Sigma(\SSS) \le \frac{29}{30}$, the final contradiction in Case~1.

\paragraph{Case 2.}
There exists $k \in I$ with $|A_k| = |B_k| = 2$.

In dealing with this case, we may assume that $|A_i|$ and $|B_i|$ are at least $3$ for every $i \in I \setminus \{k\}$, otherwise we fall back on Case~1.

Our aim in this case is to apply Lemma~\ref{lem:induction} using the reduced systems of the form $\SSS[I^A_{\bar{v}}] - \{v\}$. To this end, we need to upper bound $\Sigma(\SSS[I^A_{\bar{v}}] - \{v\})$ for the various choices of $v \in V(\SSS)$. By Lemma~\ref{lem:reduction} and the minimality of $\SSS$, we know that $\Sigma(\SSS[I^A_{\bar{v}}] - \{v\}) \le \frac{29}{30}$ as long as all sets in this reduced system have size at least $2$. This is the case when $v \in A_k$, because then $k \notin I^A_{\bar{v}}$ and the sets associated with $i \ne k$ are large enough. This is also the case when $v \notin A_k \cup B_k$, because then the removal of $v$ leaves $A_k$ and $B_k$ intact.

It remains to deal with the reduced system $\SSS[I^A_{\bar{v}}] - \{v\}$ when $v \in B_k$. This system contains the pair $(A_k, B_k \setminus \{v\})$, which we may write as
\[ A_k = \{x,y\},\,\,\, B_k \setminus \{v\} = \{z\}. \]
As the system is $1$-cross intersecting by Lemma~\ref{lem:reduction}, for any other $i \in I^A_{\bar{v}}$ the set $A_i$ must contain $z$ and the set $B_i \setminus \{v\}$ must contain exactly one of $x$ and $y$. We get a partition of $I^A_{\bar{v}} \setminus \{k\}$ into two (possibly empty) subsets $I_1$ and $I_2$ as follows:
\begin{align*}
I_1 & = \{i \in I^A_{\bar{v}}:\, x \in B_i,\, y \notin B_i,\, z \in A_i\}, \\
I_2 & = \{i \in I^A_{\bar{v}}:\, x \notin B_i,\, y \in B_i,\, z \in A_i\}.
\end{align*}
In the subsystem $(\SSS[I^A_{\bar{v}}] - \{v\})[I_1]$ neither $x$ nor $z$ appears in any cross intersection. Hence by Lemma~\ref{lem:reduction}, the reduced system $(\SSS[I^A_{\bar{v}}] - \{v\})[I_1] - \{x,z\}$ is $1$-cross intersecting. By Theorem~\ref{thm:bollobas} we have
\[ \Sigma((\SSS[I^A_{\bar{v}}] - \{v\})[I_1] - \{x,z\}) = \sum_{i \in I_1} \frac{1}{{|A_i \setminus \{z\}| + |B_i \setminus \{v,x\}| \choose |A_i \setminus \{z\}|}} \le 1. \]
Noting that $|A_i| \ge 3$ and $|B_i \setminus \{v\}| \ge 2$ for every $i \in I_1$ and using Lemma~\ref{lem:ratio}, this gives
\[ \Sigma((\SSS[I^A_{\bar{v}}] - \{v\})[I_1]) = \sum_{i \in I_1} \frac{1}{{|A_i| + |B_i \setminus \{v\}| \choose |A_i|}} \le \frac{3}{10} \sum_{i \in I_1} \frac{1}{{|A_i| + |B_i \setminus \{v\}| - 2 \choose |A_i| - 1}} \le \frac{3}{10}. \]
An analogous argument gives the bound $\Sigma((\SSS[I^A_{\bar{v}}] - \{v\})[I_2]) \le \frac{3}{10}$. Thus
\[ \Sigma(\SSS[I^A_{\bar{v}}] - \{v\}) = \frac{1}{{|A_k| + |B_k \setminus \{v\}| \choose |A_k|}} + \Sigma((\SSS[I^A_{\bar{v}}] - \{v\})[I_1]) + \Sigma((\SSS[I^A_{\bar{v}}] - \{v\})[I_2]) \le \frac{1}{3} + \frac{3}{10} + \frac{3}{10} = \frac{14}{15} < \frac{29}{30}. \]
By now we know that $\Sigma(\SSS[I^A_{\bar{v}}] - \{v\}) \le \frac{29}{30}$ for every $v \in V(\SSS)$. Lemma~\ref{lem:induction} implies that $\Sigma(\SSS) \le \frac{29}{30}$, a contradiction which completes the proof.
\end{proof}

\bibliographystyle{plain}
\bibliography{cross}

\end{document}